\documentclass[reqno,12pt]{amsart}
\usepackage{graphicx,amsfonts,amssymb,amsmath,amsthm,amscd}
\usepackage[all]{xy}
\usepackage{color}
\usepackage{hyperref}
\usepackage{anysize} \marginsize{1.3in}{1.3in}{1in}{1in}
\hypersetup{colorlinks,linkcolor=magenta,citecolor=blue,filecolor=blue,urlcolor=blue}

\theoremstyle{plain} 
\newcounter{wow}
\newtheorem{theorem}    [wow]{Theorem}
\newtheorem{lemma}      [wow]{Lemma}
\newtheorem{corollary}  [wow]{Corollary}
\newtheorem{proposition}[wow]{Proposition}

\theoremstyle{definition}
\newtheorem{definition} [subsection]{Definition}
\newtheorem*{definition*}           {Definition}

\theoremstyle{remark}
\newtheorem{remark}[wow]{Remark}

\definecolor{prpl}{rgb}{0.7, 0.0, 0.7}

    \newcommand{\mysect}[1]{
  \subsection{}\hspace{-.5em}\textbf{#1.}
\vskip .5em}

\def\Sym{\operatorname{Sym}}

\def\id{\operatorname{id}}
\def\GL{\operatorname{GL}}
\def\PGL{\operatorname{PGL}}

\def\charact{\operatorname{char}}
\def\op{\operatorname{op}}

\def\tr{\operatorname{tr}}
\def\N{\operatorname{N}}
\def\End{\operatorname{End}}
\def\SL{\operatorname{SL}}
\def\PSL{\operatorname{PSL}}
\def\disc{\operatorname{disc}}

\def\mult{\operatorname{mult}}
\def\vol{\operatorname{vol}}
\def\pt{\operatorname{pt}}
\def\arcosh{\operatorname{arcosh}}

\def\P{\mathbb{P}}
\def\Q{\mathbb{Q}}
\def\Z{\mathbb{Z}}
\def\C{\mathbb{C}}
\def\R{\mathbb{R}}

\def\F{\mathbb{F}}
\def\H{\mathbb{H}}

\def\D{\mathbb{D}}

\renewcommand{\O}{\mathcal{O}}

\renewcommand{\phi}{\varphi}

\renewcommand{\bar}[1]{\overline{#1}}

\def\into{\rightarrow}
\def\inject{\hookrightarrow}
\newcommand{\quat}[2]{\left(\frac{#1}{#2}\right)}
\newcommand{\mat}[4]{\begin{pmatrix}#1&#2\\#3&#4\end{pmatrix}}
\newcommand{\CM}{\mathrm{CM}}

 \title[Frey-Mazur conjecture over low genus curves]{On the Frey-Mazur conjecture over \\low genus curves}
 \date{\today}
 \author{Benjamin Bakker}
 \address{B. Bakker:
 Courant Institute of Mathematical Sciences, New York University,
 251 Mercer St., New York, NY 10012
 }
 \email{bakker@cims.nyu.edu}
\author{Jacob Tsimerman}
\address{J. Tsimerman:
Mathematics Department, Harvard University, 1 Oxford Street, Cambridge, MA,02138}
\email{jacobt@math.harvard.edu}

\begin{document}
\begin{abstract}
The Frey--Mazur conjecture states that an elliptic curve over $\Q$ is determined up to isogeny by its $p$-torsion Galois representation for $p\geq 17$.  We study a geometric analog of this conjecture, and show that the map from isogeny classes of ``fake elliptic curves"---abelian surfaces with quaternionic multiplication---to their $p$-torsion Galois representations is one-to-one over function fields of small genus complex curves for sufficiently large $p$ relative to the genus.
\end{abstract}
\maketitle

\section{Introduction}

The Frey-Mazur conjecture, originating in \cite{mazur}, states that for a prime $p\geq17$, an elliptic curve over $\Q$ is classified up to isogeny by its $p$-torsion, viewed as a Galois representation (or equivalently, as a finite flat group scheme). Geometrically, there is a surface $Z(p)$ that parameterizes
pairs of elliptic curves $(E,E')$ together with an isomorphism $\phi$ of their $p$-torsion, and this surface is endowed with natural Hecke divisors $H_M$ parametrizing points for which $\phi$ is induced by an isogeny of degree $M$. The conjecture is equivalent to the statement that for $p\geq 17$, all rational points of $Z(p)$ lie on one of these divisors\footnote{Note that by work of Mazur, it is only necessary to consider $M\leq 163$}. Since by work of Hermann \cite{hermann} the surface $Z(p)$ is of general type for $p>11$, the Bombieri-Lang conjecture implies that there are only finitely many rational points on the complement of the union of all rational and elliptic curves in $Z(p)$. Hence, it becomes natural to first consider the Frey-Mazur conjecture over the function fields of curves of genus at most 1.

Rather than work with elliptic curves themselves, we instead work with what are often called ``fake elliptic curves": abelian surfaces with an action by a maximal order $\O_D$ in a quaternion algebra. Such abelian surfaces are also parametrized by a one-dimensional Shimura variety $X^D$, but crucially for us these curves are \emph{compact}.  There is an obvious natural analog of the Frey-Mazur conjecture in this setting as well. Our main result is:

\begin{theorem}[see Theorem \ref{bigthm}]\label{main}  For any $k>0$, there exists $N>0$ such that for any smooth quasiprojective complex curve $B$ of genus $g<k$ and any two abelian surfaces $A_1,A_2$ over $B$ with $\O_D$-actions whose $p$-torsion local systems $A_i[p]$ are isomorphic (as $\O_D$-modules), $A_1$ and $A_2$ are $\O_D$-isogenous provided $p>N$.
\end{theorem}

The statement of Theorem \ref{main} is equivalent to the assertion that any map from a curve of genus $g<k$ to the modular surface $Z^D(p)$ analogous to $Z(p)$ above lies in a Hecke divisor $H$ for $p\gg 0$.  The proof of Theorem \ref{main} is substantially easier assuming the base Shimura curve $X^D$ has genus $\geq 2$ (see Corollary \ref{easygenus}), as the main difficulty is understanding rational and elliptic curves in $Z^D(p)$.  A catalog of low genus Shimura curves can be found in \cite{voight}.

Theorem \ref{main} is ostensibly about curves in $Z^D(p)$, however, it is easier to understand curves in the product $X^D(p)\times X^D(p)$, where $X^D(p)$ parameterizes such abelian surfaces $A$ rigidified by an $\O_D$-isomorphism $A[p]\cong \O_D[p]$---equivalently an element $z\in A[p]$ generating $A[p]$ over $\O_D[p]$.  The surface $Z^D(p)$ is naturally the quotient of $X^D(p)\times X^D(p)$ only remembering the composition $A_1[p]\xrightarrow{\cong}\O_D[p]\xrightarrow{\cong}A_2[p]$.

The main idea of the proof of Theorems \ref{main} runs as follows:  given a curve $B$ in $Z^D(p)$ we lift it to a curve $C$ in the surface $X^D(p)\times X^D(p)$ and estimate its genus using Riemann-Hurwitz in 2 different ways. We obtain a lower bound from the projections to the curves $X^D(p)$ simply by ignoring ramification. The lower bound only becomes useful once we know that the bidegree of non-Hecke curves has to be large.  We deduce this from a theorem of Andre and Deligne \cite{andre} which says that the image of the fundamental group of a non-Hecke curve is Zariski dense.  This argument alone is enough to conclude Theorem \ref{main} when $g(X^D)>1$.

The upper bound requires a bound on the ramification divisor of $C\rightarrow B$, which can only be supported at the singular points of $X^D(p)\times X^D(p)\rightarrow Z^D(p)$, so we look to bound the number of times $C$ can pass through this set. The singularities naturally split into two sets which, following Kani and Schanz \cite{kani}, we label the ``Heegner'' and ``anti-Heegner'' CM points. We show that with respect to the hyperbolic metric, the Heegner CM points are far away from each other, except for a set which lie on low degree Hecke curves. We then use work of Hwang and To \cite{hwangto1,hwangto2} to show that curves $C$ with high incidence along the Heegner CM points must have large volume. The absence of Hecke curves passing through the anti-Heegner CM points requires us to prove an analogous bound on the volume of the curve $C$ near the \emph{conjugate} Hecke curves.

\mysect{The elliptic curve case}
Since the writing of this preprint, the authors have proven \cite{BT2} the analog of Theorem \ref{main} for elliptic curves, namely that two elliptic curves over the function field of a complex curve $B$ with isomorphic $p$-torsion are isogenous provided $p$ is larger than a constant $N$.  The proof follows the same strategy as that outlined above, though the analysis is substantially complicated by the existence of cusps on the modular curves $X(p)$.  Furthermore, it is shown there that the constant $N$ can be taken to depend only on the gonality of $B$, which is the analog of the degree of a number field in the function field setting.  Though this preprint is largely subsumed and substantially generalized by \cite{BT2}, it has two advantages:\begin{itemize}
\item The Shimura case simply exhibits the core idea;
\item The argument of Section \ref{lowdegreesect} is not needed in \cite{BT2} but may still be of interest.
\end{itemize}
We note here that using the techniques of \cite{BT2}, the constant $N$ of Theorem \ref{main} can be likewise taken to depend only on the gonality of $B$ (\emph{cf}. Remark \ref{gonality}).  We also expect the method of this paper to work for all compact Shimura curves but do not pursue this here.
\mysect{Outline}\noindent
We now give an outline of the rest of the paper. In Section \ref{shimurasect} we recall background on quaternion algebras, Shimura modular curves and level structures.  We carefully treat the uniformization of these curves,  and classify the points with additional automorphisms (the ``Heegner" and ``anti-Heegner" CM points).  In Section \ref{heegnersect}, we prove that those Heegner CM points that are not well spread out lie on low degree Hecke curves, and in Section \ref{hwangtosect} we use this to bound the incidence of non-Hecke curves $C\subset X^D(p)\times X^D(p)$ along the CM points.  In Section \ref{lowdegreesect} we show that the bidegree of non-Hecke curves $B\subset Z^D(p)$ grows with $p$.  Section \ref{mainsect} contains the proof of Theorem \ref{main}.

\mysect{Acknowledgements}\noindent  The authors benefited from many useful conversations with Fedor Bogomolov, Johan de Jong, Michael McQuillan, Allison Miller, and Peter Sarnak. The first named author was supported by NSF fellowship DMS-1103982.

\mysect{Notation}\noindent

Throughout the paper we use the following notation regarding regarding asymptotic growth: For functions $f,g$ we write $f\gg g$ if there is a positive constant $L>0$ such that $f-Lg$ is a positive function; likewise for $\ll$. We may also sometimes write $f=O(g)$ instead of $f\ll g$.  If $f_t,g_t$ are functions depending on $t$, we write $f_t=o(g_t)$ as $t\into\infty$ to mean that for all $L>0$, there exists $N>0$ such that $g_t-Lf_t $ is positive, provided $t>N$.
\section{Shimura Modular Curves}\label{shimurasect}

We begin by briefly recalling the theory of Shimura modular curves
over $\Q$.  Our main reference is \cite[Chapter 4]{milne} for quaterion algebras and
\cite{shimura}, \cite{elkies} for Shimura curves.
\mysect{Quaternion algebras}\noindent
Let $k$ be a field of characteristic $\charact k\neq 2$.  Recall
that a quaterion algebra $D/k$ over a field $k$ is a central simple
algebra over $k$ with $\dim_k D=4$.  The trivial (or split)
quaterion algebra is $D=M_2(k)$, the algebra of 2 by 2 matrices over
$k$.  Given an extension $K/k$, we say $D$ is split over
$K$ if $D\otimes_k K\cong M_2(K)$, and we similarly define $D$ to be
split at a place $v$ of $k$ if $D\otimes k_v$ is split.  A quaternion
algebra $D$ over $\Q$ is indefinite if it is split at the infinite
place.

We can construct quaternion algebras analogously to the usual
Hamiltonian quaternions.  For $\alpha,\beta\in
k$ we define $\quat{\alpha,\beta}{k}$ to be the quaternion algebra with $k$-basis
$1,i,j,ij$ and relations
\[i^2=\alpha,\hspace{.5in} j^2=\beta,\hspace{.5in} ij=-ji\]
For
example, $\quat{1,1}{k}\cong M_2(k)$ is split;
$\quat{-1,-1}{\R}$ is the usual Hamiltonian quaternions.  $D=\quat{\alpha,\beta}{k}$ comes endowed with a canonical involution
$\bar{\,\cdot\,}:D\xrightarrow{\cong} D^{\op}$ given by
\[\bar{a+bi+cj+dij}=a-b i-c
j-d ij\]  With our
hypotheses on the characteristic, every quaternion algebra $D/k$ is
representable as $\quat{\alpha,\beta}{k}$ for some $\alpha,\beta\in k$.
 We define the reduced trace and norm to
be the maps
\[\tr:D\into k:x\mapsto x+\bar{x}\]
\[\N:D\into k:x\mapsto x\bar{x}\]

So for $D=\quat{\alpha,\beta}{k}$, we have that $\tr(a+b i+c j+d
ij)=2a$, and
\[\N(a+b i+c j+d
ij)=a^2-\alpha b^2-\beta c^2+\alpha\beta d^2\]

For example, the involution of $M_2(k)$ is
\[\bar{\mat{a}{b}{c}{d}}=\mat{d}{-b}{-c}{a}\]
and the reduced trace and norm are simply the trace and determinant, respectively.

Note that $D=\quat{\alpha,\beta}{k}$ is naturally split over $K=k(\sqrt{\alpha})$.  Indeed,
$K\cong k\oplus ki$ is a subalgebra of $D$, and the left regular
representation $L:D\into \End_K(D)$ mapping $x\in D$ to left
multiplication by $x$ becomes an isomorphism over $K$.  Using the
$K$-basis $1,j$, the representation is explicitly given by
\[L((a+b i)+(c+d i)j)=\mat{a+b
  i}{b(c-d i)}{c+d i}{a +b i}\]

If $k$ is a number field or a $p$-adic field, an order $\O$ of a quaternion algebra $D/k$ is a subring containing
the ring of integers $\O_k$ of $k$ which is finite as a module over $\O_k$.  For example, $M_2(\O_k)$ and
$\O_k[i,j]$ are orders in $M_2(k)$ and
$\quat{\alpha,\beta}{k}$, respectively.  For any order $\O$, we define
$\O_+^*$ to be the group of units of positive norm,
$\O^*_1\subset \O^*_+$ to be the norm 1 subgroup, and the discriminant $\disc\O$ to be its discriminant with respect to the reduced trace form.

\mysect{Shimura modular curves}  \noindent Throughout the remainder of the paper, let $D/\Q$ be a nonsplit indefinite quaternion algebra of
discriminant $d$, and
let $\O_D$ be a maximal order of $D$.  Note that because $D$ is
indefinite, all of its maximal orders are conjugate \cite[Theorem 14]{clark}.

For a variety $S$, an
abelian surface over $S$ with an $\O_D$-action is an abelian scheme $A/S$ of relative
dimension 2 with an injective
ring homomorphism $\iota:\O_D\inject \End_S(A)$.  Let $\mathcal{X}^D/\Q$ be
the stack of such families.  The associated coarse space $X^D/\Q$ is a
smooth proper curve, called a Shimura curve.

Shimura curves can be thought of loosely as generalized elliptic modular curves.
Indeed, one can view the elliptic modular curve $X(1)$ as constructed
via the above procedure by taking the maximal order $\O_D=M_2(\Z)$ in the split
quaternion algebra $D=M_2(\Q)$.  An abelian surface $A$ with
an $\O_D$-action is then forced to be the square of an
elliptic curve with the obvious inclusion $\O_D\inject \End(A)$.

For $N$ coprime to $d$, we define $X^D_0(N)/\Q$ to be the coarse space
associated to the stack of abelian surfaces with an
$\O_D$-action together with a $\Z/N\Z$-rank 2 (left) $\O_D$-submodule $V$ of the $N$-torsion
$A[N]$.  $X^D_0(N)$ is a smooth proper curve and admits two maps:  $\mu:X^D_0(N)\into X^D$ forgetting the torsion submodule,
and $\nu:X^D_0(N)\into X^D$ sending $A$ to $A/V$. Since $N$ is coprime to $d$, $D$ splits over $\Q_p$ for each $p|N$, and therefore $\O_D(\Z/N\Z)\cong M_2(\Z/N\Z)$. Thus there are $\prod_{p^e||N} (p^e+p^{e-1})$ such modules $V$.  By analogy
with the elliptic modular curve case, we say that $A/V$ is cyclically
isogenous to $A$.

 Again for $N$ coprime to $d$, a full level $N$ structure on an abelian surface $A$ with an $\O_D$ action is an isomorphism $A[N]\cong \O_D[N]$ of $\O_D$-modules, and two level structures are equivalent if the isomorphisms are equal up to scale  (\emph{cf.} Remark \ref{level}).  Equivalently, a full level $N$ structure is an element $v\in A[N]$ such that $\O_D[N]v=A[N]$, defined up to scaling by $(\Z/N)^*$.  The Shimura curves with full level structure $X^D(N)$ are then defined as the coarse space associated to the
 stack $\mathcal{X}^D(N)$ of abelian surfaces $A$ with an $\O_D$-action and
 endowed with full level $N$ structure.

 There is similarly an
 obvious forgetful map $\pi: X^D(N)\into X^D$, and for any $M$ coprime to $N$, the two maps
 $X^D_0(M)\into X^D$ induce a Hecke correspondence $T_M$ (where we drop the $N$ by abuse of notation):
\[\xymatrix{
&\ar[ld]T_M\ar[dd]\ar[rd]&\\
X^D(N)\ar[dd]_\pi & & X^D(N)\ar[dd]^\pi\\
&\ar[ld]_\mu X^D_0(M)\ar[rd]^\nu&\\
X^D&&X^D\\
}\]

Explicitly, points in the image $T_M\into X^D(N)\times X^D(N)$ are
pairs of cyclically isogenous abelian surfaces with an $\O_D$-action
 and full level $N$ structure such that the isogeny is of degree $M$ and
induces an isomorphism of level $N$ structures.  By the above, the degree of $\mu$ and $\nu$ is $\prod_{p^e||N} (p^e+p^{e-1})$.

\mysect{Uniformization}
Like elliptic modular curves, Shimura curves can also be represented explicitly as quotients of $\H^{\pm}=\C\backslash\R$
by discrete groups of isometries.

Define $\Gamma^D=\O_{D}^*/\pm 1$.  Since $D$ is indefinite, we may
choose an isomorphism $\phi_D: D\otimes\R\xrightarrow{\cong} M_2(\R)$, and under this
isomorphism $(D\otimes\R)^*\cong \GL_2(\R)$. Moreover, this induces an
inclusion $\Gamma^D\inject\PGL_2(\R)$ as a discrete cocompact subgroup, and in fact we have the following
\begin{lemma}\label{unieasy}
$X^D(\C)\cong\Gamma^D\backslash\H^{\pm}$.
\end{lemma}

\begin{proof}

We can make the above isomorphism explicit. For $z\in\H^{\pm}$ set $$L_z:=\phi_D(\O_D)\cdot(1,z)\subset \C^2$$ and set
$A_z:=\C^2/L_z$. This is a complex torus with an $\O_D$-action given by $\phi_D$, and by choosing $\mu\in\O_D$ such that $\mu^2=-\rm{disc}(D)$, $A_z$ can be given the structure of an abelian surface by the Riemann form $(x(1,z),y(1,z))= \tr(\mu x\bar{y})$ for $x,y\in D\otimes\R$.
It is easy to check that acting on $z$ by $\O_D^*$ preserves the lattice $L_z$ up to a right $\O_D^*$-action. Thus we have a well defined map
$$\psi:\Gamma^D\backslash\H^{\pm}\rightarrow X^D(\C), \hspace{.25in}\psi(z)=(A_z,\phi_D)$$

Likewise, given an abelian surface $A/\C$ with an $\O_D$-action $\iota$, we can pick a uniformization $A(\C)\cong \C^2/L$ in such a way that the induced $\O_D$-action on $\C^2$ is given
by $\phi_D$. Then, as the Picard group of $\O_D$ is trivial by \cite[Theorem 15]{clark} we can pick an element $(v,w)\in L$ which generates $L$ over $\O_D$, and this element is unique up to the $\O_D^*$-action. Setting $\xi(A) = z$ gives us a well defined map from $X^D(\C)$ to $\Gamma^D\backslash\H^{\pm}$ and it is easy to check that $\xi$ and $\psi$ are inverse to each other.

\end{proof}

Note that $X^D(\C)$ can have 2 connected components. In fact, this will be the case precisely when $\O_D^*$ has an element of norm $-1$, as can easily be deduced from the above.

In uniformizing $X^D(p)(\C)$, it turns out to be convenient to use two copies of $\H^{\pm}$ instead of one. Briefly, the reason for this is that the square class of the Weil pairing of any two particular $p$-torsion elements is an invariant, and thus $X^D(p)(\C)$ has twice as many connected components as $X^D$.

Define \[\Gamma^D(p):=\ker(\Gamma^D\into (\O_D\otimes\F_p)^*/\pm 1)\]
Fix a non-square element $\alpha\in \F_p^*$ and an element $g_0\in \O_D\otimes\F_p$ such that $g_0^2 = \alpha$. Note that $g_0$ exists because
$\F_p^2$ embeds into $M_2(\F_p)$ which is isomorphic to $\O_D\otimes \F_p$ for $p\gg 1$. We define two maps, $\psi_1,\psi_2:\H^{\pm}\into X^D(p)$.  Both maps send $z$ to the abelian surface $(A_z,\phi_D)$ as in the proof of Lemma \ref{unieasy}, but $\psi_1$ assigns the torsion element
$\frac{z}{p}$ whereas $\psi_2$ assigns the torsion element $\phi_D(g_0)\left(\frac{z}{p}\right)$. It is easy to see that both these maps are injective and well defined on
$\Gamma^D(p)\backslash\H^{\pm}$. To set notation, we label the sources of $\psi_i$ as $\H^{\pm}_i$ for $i=1,2$.

Now, note that the monodromy group of $X^D(p)/X^D$ is $G_p:= \P(\O_D\otimes\F_p)^*$ acting naturally on the $p$-torsion element.

\begin{lemma}\label{unihard}
\begin{enumerate}
\item[(a)] If $\O_D^*$ has an element of norm $-1$ and $-1\notin(\F_p^*)^2$ then $$X^D(p)(\C)\cong \Gamma^D(p)\backslash \H^{\pm},\hspace{0.5in}\O_D^*/\Gamma^D(p)\cong G_p$$ and the action of $G_p$ is induced by the $\phi_D$ action of $\O_D^*$ on $\H^{\pm}$.

\item[(b)] Else, $$X^D(p)(\C)\cong\Gamma^D(p)\backslash\H_1^{\pm}\cup \Gamma^D(p)\backslash\H_2^{\pm}$$ where the monodromy action is given as follows:
$g_0\in G_p$ acts as $(z,w)\rightarrow (w,z)$, while $g\in(\O_D\otimes \F_p)^*_1$ acts as $(z,w)\rightarrow (g z, g_0^{-1}gg_0w)$, as induced by the $\phi_D\times \phi_D$ action of $\O_D^*$.
\end{enumerate}
\end{lemma}

\begin{proof}

First, we show that the union of $\psi_1$ and $\psi_2$ is surjective. If we have an abelian variety with an $\O_D$-action $(A_z,\iota_z)$ with a non-zero $p$-torsion element $v$, then $v=hz$ for some $h\in (\O_D\otimes\F_p)^*$. Since $\SL_2(\F_p)\subset (\O_D\otimes\F_p)^*$ we have $$\P(\O_D\otimes\F_p)=(\O_D\otimes\F_p)^*\cup g_0(\O_D\otimes \F_p)^*$$ where the tilde denotes reduction mod $p$. The surjectivity follows. Moreover, if $\O_D^*$ has an element of norm $-1$ and $-1\notin(\F_p^*)^2$ then we likewise see that $\psi_1$ is surjective. In this case, the rest of (a) follows easily.

For (b), we must only point out that $g\cdot g_0z = g_0\cdot g_0^{-1}gg_0 z$, and that $g_0^2$ acts as a constant on $p$-torsion by construction. The proof follows similarly to above.

\end{proof}


\mysect{Heegner and anti-Heegner CM points}
\noindent
Suppose $z\in X^D(p)$ is a point with a non-trivial stabilizer in $G_p$. Then $z$ corresponds to a pair $(A,\iota:\O_D\hookrightarrow E=\End(A))$ that has an automorphism
besides $\pm1$.

\begin{lemma}
For $z$ with non-trivial stabilizer in $G_p$, $A_z$ is isogenous to $E_i\times E_i$ or $E_{\omega}\times E_{\omega}$, where we write $E_z$ for the elliptic curve $\C/\langle 1,z\rangle$.
\end{lemma}
\begin{proof} By assumption, there is an element $x\in E^*$ which commutes with all of $\O_D$. Now, $E\otimes\Q$ must be a central simple algebra of rank 8, and thus is a matrix algebra
over a quadratic field $K=Z(E\otimes\Q)$. Let $R=K\cap E$. Since $D$ is rank 4 over $\Q$, $D$ and $K$ generate all of $E$ over $\Q$. Thus, the commutant of $D$ in $E\otimes\Q$ is the center of $E$, and thus the commutant of $\O_D$ in $E$ is $R$. Since $R$ is a quadratic ring, it follows that the order of $x$ is either 2 or 3, and $R=\Z[i]$ or $R=\Z[\omega]$. We claim that $E\otimes\Q\cong M_2(K)$. Indeed, if $A$ were simple and had a 4-dimensional CM field acting on it, then that would be its entire endomorphism algebra.
Thus $A$ is isogenous to $E_i\times E_i$ or $E_{\omega}\times E_{\omega}$, as desired.
\end{proof}

 We let $A_i$ and $A_{\omega}$ to be fixed abelian varieties isogenous to $E_i\times E_i$ and $E_{\omega}\times E_{\omega}$ respectively, together with an $\O_D$-action.
Let $R_i,R_{\omega}$ be fixed subrings of $\O_D$ isomorphic to $\Z[i],\Z[\omega]$ respectively, assuming these exist. Then we have the more refined

\begin{lemma}\label{ns}
For $z$ with non-trivial stabilizer in $G_p$, $A_z$ is $\O_D$-isogenous to $A_i$ or $A_{\omega}$.
\end{lemma}

\begin{proof}

We assume $A_i$ is isogenous to $A_z$, the other case being analogous. We can consider their complex points as $A_i(\C)=\C^2/L_1$ and $A_z(\C)=\C^2/L_2$ respectively, where $L_1$ and $L_2$ are commensurable. Let $V=L_1\otimes\Q=L_2\otimes\Q$.

Let $E_1$ and $E_2$ be the endomorphism rings of $A_i$ and $A_z$ respectively. Then we can identify $E_1\otimes\Q$ with $E_2\otimes\Q$ as the complex endomorphisms of $\C^2$ that preserve $V$. Now, let $\iota_1,\iota_2$ be the two embeddings of $\O_D$ into $E$ corresponding to $A_i,A_z$ respectively. Since $E_1\otimes\Q$ is a central simple algebra, there exists an element $e$
such that $e\iota_1e^{-1}=\iota_2$.  Considering a sufficiently large integer $n$ such that $neL_1\subset L_2$, we see that the map $z\rightarrow nez$ is an $\O_D$ isogeny from
$A_i$ to $A$. This completes the proof.
\end{proof}

Now, suppose that $z,w\in X_D(p)$ have a common stabilizer  $g\in G_p$. Then by the discussion above, the centers of their rings of endomorphisms are $R=\Z[i]$ or $R=\Z[\omega]$,
and there is a lift $g'\in (\O_D\otimes\F_p)^*$ and $x_z,x_w\in R^*$ such that $g' v_z=x_zv_z, g' v_w=x_wv_w$ for $v$ the chosen $p$-torsion element. Note that $g'$ must have reduced norm $1$, and so is determined up to negation, and in particular we can ensure it has order $4$ or $3$. Once we pick such a $g'$, it determines $x_z,x_w$. Then on the complex tangent space at 0 of $A_z$, $x_z$ acts as one of two scalars: either $\{i,-i\}$ or $\{\omega,-\omega\}$.

\begin{definition}\label{conjugate} Generalizing the notation of \cite{kani}, if the eigenvalues of $x_z$ and $x_w$ acting on the tangent spaces at 0 are the same, we say that $(z,w)$ is a Heegner CM point. Else, the eigenvalues differ by conjugation and we say $(z,w)$ is an anti-Heegner CM point.
\end{definition}

We remark here that $(z,w)$ is a Heegner CM point if and only if $(\bar{z},w)$ is an anti-Heegner CM point.
\mysect{Diagonal quotient varieties}
\noindent
In this section we introduce
the main objects of study, the diagonal quotient varieties $Z^D(p)$.  The stack $\mathcal{Z}^D(p)$ is the stack of pairs of abelian surfaces $(A_1,A_2)$ with $\O_D$-actions, together with a $\O_D$-isomorphisms up to scale
$A_1[p]\xrightarrow{\cong}A_2[p]$.  $\mathcal{Z}^D(p)$ is the
stack quotient $[G\backslash\mathcal{X}^D(p)\times\mathcal{X}^D(p)]$, where $G=\Gamma^D/\Gamma^D(p)$ acts
diagonally.  We let $Z^D(p)$ be the coarse space associated to
$\mathcal{Z}^D(p)$.  Note that $Z^D(p)$ is simply the scheme quotient
$G\backslash X^D(p)\times X^D(p)$, again with $G$ acting diagonally.  The
variety $Z^D(p)$ is therefore a proper projective normal scheme defined over $\Q$ with cyclic quotient singularities of order 2 or 3. The
diagonal quotient surface obtained as above from elliptic modular
curves has been studied in detail by \cite{hermann} and \cite{kani}.

 \begin{remark}\label{level} Recall that our definition of a full level structure is only up to scale, and thus the $\O_D$-isomorphisms parameterized by $\mathcal{Z}^D(p)$ are up to scale. For the problem we are considering, it would be more natural for the notion of isomorphism of level structures to coincide with that of $\O_D$-modules, but we prefer our approach as it avoids keeping track of the Weil pairing and notationally distinguishing different connected components of the Shimura curves. We note that there is a forgetful map from the unscaled level structure to the one we are considering, so in actuality we are proving a slightly stronger theorem.
 \end{remark}

The variety $Z^D(p)$ comes equipped with Hecke curves $H_M$ for any
$M$ coprime to $p$.  $H_M$ is the image of the Hecke correspondence $T_M\into
X^D(p)\times X^D(p)$ from Section 2.2 under the quotient map and parametrizes pairs of
abelian surfaces $(A_1,A_2)$ with an isomorphism up to scale
$A_1[p]\xrightarrow{\cong}A_2[p]$ induced by a cyclic isogeny $A_1\into
A_2$ of degree $M$. 
\section{Geometry of the Heegner CM points}\label{heegnersect}
\mysect{Preliminaries}\noindent
Throughout this section, the complex points of our modular curves $X$ will be equipped with canonical metrics of constant sectional curvature $-1$ inherited from $\H^{\pm}$.  Henceforth we will typically blur the notational distinction between $X$ and $X(\C)$.  For $x\in X$ we let $B(x,r)$ be the set of all points in $X$ within distance $r$ of $x$.

Recall that the injectivity radius $\rho_X(x)$ of $X$ at a point $x\in X$ is half the length of the shortest closed geodesic through $x$.  Equivalently, it is the radius $r$ of the largest isometrically embedded hyperbolic ball $B(x,r)\subset X$.  The injectivity radius $\rho_X$ is then the infimum of $\rho_X(x)$ over all $x\in X$, or equivalently half the length of the shortest closed geodesic in $X$.  It was first observed by Buser-Sarnak in \cite{BuSa} that the injectivity radii of Shimura curves are large. For the convenience of the reader, we recall the proof:

\begin{lemma}\label{injrad}$\rho_{X^D(p)}\geq 2\log p+O(1)$.
\end{lemma}
\begin{proof}

A closed geodesic through $x\in X^D(p)$ lifts to the unique geodesic arc between two lifts $z,\gamma z\in\H^{\pm}$ for some non-identity $\gamma\in \Gamma^D(p)$.  Note that because $\Gamma^D(p)$ is cocompact, every element of $\Gamma^D(p)$ is semisimple, so $d(z,\gamma z)=d(Az,(\tr \gamma)Az)$, where $A\in \SL_2\R$ is the diagonalizing matrix. In particular, using the formula for distance in the upper half-plane, this means
\begin{align*}
\min_{z}d(z,\gamma z)&=\min_z d(z,az)\notag\\
&=\min_z \arcosh\left(1+\frac{(a-1)^2|z|^2}{2a(\Im z)^2}\right)\notag\\
&\geq \arcosh\left(\tr(\gamma^2)/2\right)
\end{align*}
Again because $\Gamma^D$ is cocompact, the only element with trace 2 is the identity, and thus the minimal value of $|\tr \gamma|$ for $1\neq \gamma\in \Gamma^D(p)$ is $2+p^2$ as
$\tr(\gamma) \cong 1\mod{p^2}$. Thus, $\tr(\gamma^2)/2=\tr(\gamma)^2/2 -1\geq p^4/2$, and the result follows.
\end{proof}
\mysect{Repulsion of Heegner CM points}\noindent
The technical heart of this section is the following
\begin{proposition}\label{repulsion}
For each $r>0$, there exists $d>0$ such that for all sufficiently large $p$, if $(x,y),(x',y')$ are distinct CM Heegner points in $X^D(p)\times X^D(p)$ with
the same projections to $X^D\times X^D$ and $B(x,r)\cap B(x',r)\neq \varnothing$ and $B(y,r)\cap B(y',r)\neq \varnothing$, then one of $(x,y),(x',y')$ lies on some Hecke divisor $T_k$ with $k<d$.
\end{proposition}

\begin{proof}

We first reduce to the case where $x,y$ project to the same point in $X^D$ (and hence so do $x',y'$). By Lemma \ref{ns} we can find an isogeny of bounded degree $B$ between the images of $x,y$ in $X^D$. Since the Hecke correspondences are locally geodesic and the pullback of $T_M$ along $T_B$ is a union of $T_k$ for $k<MB$, by pulling back along $T_B$ in the first co-ordinate we can assume that $x$ and $y$ project to the same point in $X^D$, at the cost of scaling $d$ by some bounded amount $B$.

Since $(x,y)$, $(x',y')$ lie in the same component of $X^D(p)$, there must be $i,j\in\{1,2\}$ and lifts $(z,w)(z',w') \in \H^{\pm}_i\times \H^{\pm}_j$ such that $(z,w)$ is in the $\O^*_{D,1}\times \O^*_{D,1}$ orbit of $(z',w')$. Replacing $(x,y)$ by $(g_0x,g_0y)$ if necessary, we can assume that $i=1$. Moreover, we can pick these lifts so that
\[d((z,w),(z',w'))=d((x,y),(x',y'))<4r\]
and by acting diagonally by $G_p$ we can pick $z$ from a finite fixed set of points $S$ independently of $p$. Set $w=gz$, $g\in G_p$.

Let $t\in \O^*_{D,1}$ be a stabilizer of $z$. Since $w$ must also be stabilized by $t$, $gtg^{-1}$ is either $t$ or $t^{-1}$ in $G_p$.  This is true in $(\O_D\otimes\F_p)^*$ up to a sign \emph{a priori}, but by comparing the traces $gtg^{-1}$ must in fact be $t$ or $t^{-1}$ in $\O_D\otimes\F_p)^*$, for any lift of $g$.  Since $(z,w)$ is a Heegner CM point by assumption, it follows that $gt=tg$.

Next, fix the set $S_r\subset\O_{D,1}^*$ of elements $\gamma$ such that $d(\gamma z,z)<r$ . Then there must exist elements $h_z,h_w\in S_r$ such that
$h_zz=z',w'=gh_wz=gh_wh_z^{-1}z'$. Hence, as in the above we must have that $gh_wh_z^{-1}$ commutes with $h_zth_z^{-1}$, or equivalently $h_zgh_w$ commutes with $t$.

The two relations $gt=tg$ and $h_zgh_wt=th_zgh_w$ can be interpreted as a set of linear equations in the coefficients of $g$, where we view $g$ as a $2\times 2$ matrix. The relation $gt=tg$ defines the field generated by $g$, so it has a 2-dimensional set of solutions. Thus, either the 2 relations define a line, or the second relation is redundant. Note that the relation is redundant over $\Q$ if and only if it is redundant over all sufficiently large finite fields.

If the second relation is redundant over $\Q$, setting $H$ to be the centralizer of $t$ in $D^*$, we must have $h_zHh_w=H$, which implies $$h_zHh_z^{-1}=(h_zHh_w)(h_zHh_w)^{-1}=HH^{-1}=H.$$

Likewise, $h_wHh_w^{-1} = H$. Now, we claim that the elements of the normalizer of $H$ in $D^*$ which have positive norm consist exactly of $H$. To prove this, note that its enough to check it after tensoring with $\R$, in which case $D$ become $M_2(\R)$ and $H$ becomes an embedded $\C^*$  (unique up to conjugation). Thus, $h_z,h_w\in H$.
Finally, note that since $H\cap \O_D$ is isomorphic to either $\Z[i]$ or $\Z[\omega]$ we must have $h_z,h_w$ stabilizers of $z$, contradicting the assumption that our Heegner CM points were distinct.

Thus, the two relations must not be redundant, and we end up with a single projective solution $g$ which we can take to be in $\O_D$. Thus $z$ and $w$ lie on $T_M$ where $M=\textrm{N}(g)$.
Setting $d$ to be bigger then all (finitely many) $M$ arising in this way gives the result.
\end{proof}

\section{CM incidence estimates}\label{hwangtosect}\noindent
For any hyperbolic curve $X$ and any curve $C\subset X\times X$, work of Hwang and To shows that the multiplicity of $C$ at a point $x\in X\times X$ is bounded in terms of the volume of $C$ in a geodesic ball centered at $x$, and similarly its intersection with a totally geodesic curve $H\subset X\times X$ is bounded by the volume of $C$ in a geodesic tubular neighborhood of $H$.
Note that the volume of a curve can be interpreted as the degree of the restriction of the canonical divisor $K_{X\times X} $ to $C$, and therefore as the intersection number $C\cdot K_{X\times X}$. This allows one to deduce algebro-geometric results from the hyperbolic ``spread-outedness" of $X$.

\begin{theorem}\label{HT}  For any curve $C\subset X\times X$, we have
\begin{enumerate}
\item[(a)]\cite[Theorem 1]{hwangto1}  For any point $x\in X$, let $B_r=B(x,r)$ for $r<\rho_X(x)$.  Then
\[\vol(C\cap B)\geq 4\pi\sinh^2\left(\frac{r}{2}\right)\mult_x(C)\]
\item[(b)]\cite[Theorem 1]{hwangto2}  Let $\Delta \subset X\times X$ and $W_r=\{(z,w)\in X\times X|d(z,w)<r\}$ for $r<\rho_X$.  Then
\[\vol(C\cap W_r)\geq 8\pi\sinh^2\left(\frac{r}{4}\right)(C\cdot\Delta)\]
\end{enumerate}
\end{theorem}
Recall from Section 2.2 that for any $m$ there are two measure preserving maps $\mu,\nu:T_m\into X^D(p)$ in the sense that push-forward of multisets preserves volume and pull-back of multi-sets multiplies volume by the degree.  The maps $X^D(p)\times T_m\into X^D(p)\times X^D(p)$ given by $\phi=\id\times\mu$ and $\psi=\id\times\nu$ then have the same property.  Letting $W^m_r=\psi_*\phi^*W_r$, we can generalize part (b) of Theorem \ref{HT} to Hecke curves:
\begin{corollary}\label{heckeHT}  For any Hecke curve $T_m\subset X^D(p)\times X^D(p)$, and $r<\rho_{X^D(p)}$
\[\vol(C\cap W^m_r)\geq 8\pi\sinh^2\left(\frac{r}{4}\right)(C\cdot T_m)\]
\end{corollary}
\begin{proof}
\begin{align*}
\vol(C\cap W^m_r)&=\vol(\phi_*\psi^*C \cap W_r)\\
&\geq \sinh^2\left(\frac{r}{4}\right)(\phi_{*}\psi^*C\cdot\Delta)\\
&= \sinh^2\left(\frac{r}{4}\right) (C\cdot\psi_{*}\phi^*\Delta)\\
&= \sinh^2\left(\frac{r}{4}\right) (C\cdot T_m)
\end{align*}
\end{proof}

Both statements in Theorem \ref{HT} are optimal in the sense that the bound is realized by a union of translates of the image of the graph of $-z:\D\into\D$.  For the convenience of the reader, we summarize the proof of part (b) above. Recall the following

\begin{definition*} For $\phi(z)$ a plurisubharmonic function on a neighborhood of a point $x$ in some complex manifold $M$, the Lelong number of $\phi$ at $x$ is
$$\nu(\phi,x):=\liminf_{z\rightarrow x} \frac{\phi(z)}{\log|z - x|}.$$
\end{definition*}
For example, if $V\subset M$ is a divisor cut out locally by $f$, and $\phi(z)=\log|f(z)|$, then $\nu(\phi,x)=\mult_x(f)$.

  Hwang and To define a plurisubharmonic function $F:\D\times\D\into\R$ that is diagonally-invariant under the full isometry group $\PSL_2\R$ such that $0\leq \omega_F=i\partial\bar\partial F \leq \omega_{std}$, as well as diagonally-invariant functions $f_\epsilon:\D\times\D\into \R$ that
\begin{enumerate}
\item are plurisubharmonic off the diagonal $\Delta_\D\subset\D\times\D$;
\item agree with $F$ outside of $B(\Delta_\D,r)$;
\item have a logarithmic pole along the diagonal, and for any $\xi\in \Delta_\D$, $$\liminf_{\epsilon\into 0}\nu(f_\epsilon,\xi)=8\sinh^2(r/4)$$
\end{enumerate}

As $f_{\epsilon}$ and $F$ descend to functions on $X\times X$, it then follows that for any curve $C\subset X\times X$,
\[\vol(C\cap B(\Delta,r))\geq \int_{C\cap B(\Delta,r)}\omega_F=\int_{C\cap B(\Delta,r)}\omega_{f_\epsilon}\geq \pi\sum_{\xi\in C\cap \Delta}\nu(f_\epsilon,\xi)\mult_\xi C\]
where we've used the diagonal invariance to descend the forms to $X\times X$.  The equality follows from Stoke's theorem and the second inequality from the fact that, for $[C]$ denoting the current of integration along $C$,
\[\nu([C]\wedge \omega_{f_\epsilon},\xi)\geq \nu([C],\xi)\nu(f_\epsilon,\xi)\]
(\emph{cf} \cite{hwangto1}, Proposition 2.2.1(a)).

We will need a result comparing the volume within a radius $R$ to that within a smaller radius $r$ of the \emph{conjugate} diagonal in order to handle the anti-Heegner CM points.  For a curve $X$, its conjugate $\bar X$ is the same curve with the negated complex structure, and the pointwise diagonal $\bar \Delta\subset X\times \bar X$ is called the conjugate diagonal.

\begin{proposition}\label{htad}  For $X$ a compact hyperbolic complex curve, any complex curve $C\subset X\times \bar X$  that is not the conjugate diagonal, and any $\rho_X>R>r>0$,
\[\vol(C\cap B(\bar\Delta,R))\geq \frac{\sinh(R/2)}{\sinh(r/2)}\vol(C\cap B(\bar\Delta,r))\]

Furthermore, the bound is optimal in the following sense: suppose $X$ is isomorphic to $\bar{X}$ via a map $z\rightarrow\bar{z}$. (For instance, $X$ is defined over $\R$). Then the graph $(z,\bar{z})$ achieves the bound.
\end{proposition}
\begin{proof}The proof is very similar to Proposition \ref{htd}.  Suppose $X=\Gamma\backslash{\mathbb{D}}$, so that $\bar X=\bar{\Gamma}\backslash{\mathbb{D}}$, and consider the function $\psi$ on $\mathbb{D}\times\mathbb{D}$ given by
\[\psi(z,w)=\tanh^2 (d_{\mathbb{D}}(z,\bar w)/2)=\left| \frac{\bar w-z}{1-zw}\right|^2.\]
$\psi$ is invariant under the diagonal action of $\SL_2\R$.  For any function $f:\R\into\R$, we compute that at $(0,w)$, the potential $F(z,w)=f(\psi)$ yields a form
\begin{align*}
\omega_F&=i\partial\bar\partial F\\
&=f'(\psi)\mat{(1-|w|^2)^2}{w^2}{\bar w^2}{1}+f''(\psi)\mat{|w|^2(1-|w|^2)^2}{-w^2(1-|w|^2)}{-\bar w^2(1-|w|^2)}{|w|^2}
\end{align*}
Taking $s(\psi)=-\log (1-\psi)$ and $S(z,w)=s(\psi(z,w))$, for instance, we have by direct computation
\[\omega_{S}=\mat{1}{0}{0}{(1-|w|^2)^{-2}}=\frac{\omega_{std}}{2}\]
 where $\omega_{std}$ is the standard form on ${\mathbb{D}}\times{\mathbb{D}}$.  Let $C= s(\tanh^2(R/2))$, $c=s(\tanh^2(r/2))$, and define a continuous function $f:[0,1]\into\R$ on the interval $\left[\tanh^2(r/2),\tanh^2(R/2)\right]$ by $f(\psi)=h(s(\psi))$, where
\[h'(s)=\frac{1-\sqrt{\frac{e^c-1}{e^s-1}}}{1-\sqrt{\frac{e^c-1}{e^C-1}}}\]
Take $f$ to be constant on $[0,\tanh^2(r/2)]$, and linear of slope 1 on $[\tanh^2(R/2),1]$.  One can easily compute that the resulting $\omega_F$ is positive, and dominated by
\[(h'(s)+2(1-e^{-s})h''(s))\frac{\omega_{std}}{2}= \left(1-\sqrt{\frac{e^c-1}{e^C-1}}\right)^{-1}\frac{\omega_{std}}{2}\]
on (the interior of) $B(\bar\Delta_\mathbb{D},R)-B(\bar\Delta_\mathbb{D},r)$.  Further we have that
$$\omega_F|_{B(\bar\Delta_\mathbb{D},r)}=0\textrm{\;\;\;and\;\;\;}\omega_F|_{B(\bar\Delta_\mathbb{D},\rho_X)-B(\bar\Delta_{\mathbb{D}},R)}=\frac{\omega_{std}}{2}$$  Smoothing $F$ out by the same trick as in the proof of Proposition \ref{htd} and descending these forms down to $X\times\bar X$, we have that
\begin{align*}
\vol(C\cap B(\bar\Delta,R))&=\int_{C\cap B(\bar\Delta,R)}\omega_{std}\\
&=2\int_{C\cap B(\bar\Delta,R)}\omega_F\\
&\leq \left(1-\sqrt{\frac{e^c-1}{e^C-1}}\right)^{-1}\left(\vol(C\cap B(\bar\Delta,R))-\vol(C\cap B(\bar\Delta,r))\right)
\end{align*}
yielding the statement, as $e^c-1=\sinh^2(r/2)$, and likewise for $C$ and $R$.
\end{proof}

Let $\CM^+$ be the set of Heegner CM points on $X^D(p)\times X^D(p)$, and $\CM^-$ the set of anti-Heegner CM points.  To get a upper bound for the genus of $C$ we will need an estimate for the total multiplicities $\mult_{\CM^{\pm}}(C)=\sum_{x\in\{\CM^{\pm}\}}\mult_x(C)$.
\begin{proposition} \label{ramification} For any non-Hecke curve $C\subset X^D(p)\times X^D(p)$, we have \[\mult_{\CM^+} (C)=o(C\cdot K_{X\times X})\] as $p\rightarrow\infty$.
\end{proposition}
\begin{proof}  Fix some $R>0$.  For $d>0$, partition $\CM^+$ into two sets \[T:=\CM^+\cap\cup_{m<d}T_m\] and $S:=\CM^+\backslash T$. By Proposition \ref{repulsion}, if $d$ is large enough in relation to $R$,  the balls $B(z,R)$ are disjoint as $z$ varies over $S$. By Theorem \ref{HT}, Lemma \ref{injrad}  and Corollary \ref{heckeHT} we then have that
\begin{align*}
\mult_{\CM^{+}}(C)&= \sum_{x\in S}\mult_x(C) + \sum_{x\in T}\mult_x(T)\\
&\ll \sinh^{-2}(R/2)\vol\left(C\cap \cup_{x\in S} B(x,r)\right) + \sum_{m<d} (C.T_d)\\
&\ll \sinh^{-2}(R/2)\vol(C) + \sinh^{-2}(p/2)\sum_{m<d}\deg T_m\vol(C)\\
&\ll (C.K_{X\times X})(\sinh^{-2}(R)+d^3\sinh^{-2}(p/2))\\
\end{align*}

As $p\rightarrow\infty$, we see that $\mult_{\CM^{+}}(C)\ll (C\cdot K_{X\times X})(\sinh^{-2}(R)+o(1))$. Since $R$ can be chosen arbitrarily large, the claim follows.

\end{proof}

\begin{proposition} \label{ramificationminus} For any non-Hecke curve $C\subset X^D(p)\times X^D(p)$, we have \[\mult_{\CM^-} (C)=o(C\cdot K_{X\times X})\] as $p\rightarrow\infty$. 
\end{proposition}
\begin{proof}

Fix some $R>0$.  We would like to perform the same trick for the anti-Heegner CM points, and for appropriately chosen $d$ we again partition the points of $\CM^-$ into
 \[T=\CM^-\cap\cup_{m<d} \bar T_m\hspace{.2in}
\mathrm{and}\hspace{.2in} S=\CM^--  T\]
where $\bar\cdot$ denotes complex conjugation on the second factor.  It will still be the case that balls of radius $R$ around points of $S$ (with $d$ chosen sufficiently large) are disjoint, but the multiplicity of a curve $C$ along $\bar T_m$ does not quite make sense, so we adjust the argument slightly:
%

\begin{align}
\mult_{\CM^{-}}(C)&= \sum_{x\in S}\mult_x(C) + \sum_{x\in T}\mult_x(C)\notag\\
&\ll \sinh^{-2}(R/2)\vol\left(C\cap \cup_{x\in S} B(x,r)\right) + \sum_{m<d} \mult_{\CM^-\cap \bar T_m} (C)\notag\\
&\ll \sinh^{-2}(R/2)\vol(C) + \sum_{m<d}\deg T_m\cdot \mult_{\CM^-\cap\bar \Delta}(T_m^* C)\label{last}
\end{align}

Since $R$ can be taken arbitrarily large, it suffices to show that for a \emph{fixed} $m$, $$\mult_{\CM^{-}\cap \bar \Delta}(T_m^* C)=o(\vol(C))$$
Because the injectivity radius is $2\log p$, there are $O_R(1)$ many overlaps of balls of radius $R$ centered at anti-Heegner $\CM$ points on the conjugate diagonal, so by Theorem \ref{HT} we have
\begin{align*}
\mult_{\CM^-\cap \bar\Delta}(T_m^*C)&\ll \sinh^{-2}(R/2)\sum_{x\in \CM^-\cap\bar \Delta}\vol(T_m^*\cap B(x,R))\\
&\ll O_R(1)\cdot \sinh^{-2}(R/2)\vol(T_m^*C\cap B(\Delta,R))\\
&\ll O_R(1)\cdot \sinh^{-1}(\log p)\vol(T_m^* C)
\end{align*}
where we've used Proposition \ref{htad} (and Lemma \ref{injrad}) in the last step.  Since $\vol(T_m^*C)=\deg T_m\vol(C)$, the Proposition is proven.
\end{proof}
Therefore, writing 
\[\mult_{\CM}C=\mult_{\CM^+}C+\mult_{\CM^-}C\]
we have 
\begin{corollary}\label{ramo}For any non-Hecke curve $C\subset X^D(p)\times X^D(p)$, we have \[\mult_{\CM} (C)=o(C\cdot K_{X\times X})\] as $p\rightarrow\infty$. 
\end{corollary}

\section{Low degree curves}\label{lowdegreesect}
Let $F=X^D\times\pt+\pt\times X^D$ be the sum of the fiber divisors on $X^D\times X^D$, and similarly $F_p=X^D(p)\times \pt+\pt\times X^D(p)$.  Likewise, let $F_Z$ be the pullback of $F$ to $Z^D(p)$ via the quotient map $q:Z^D(p)\to X^D\times X^D$.  Note that for any map from a curve $B\into Z^D(p)$, $B.F_Z$ is simply the sum of the degrees of the two maps $B\into X^D$.
\begin{proposition}\label{degreegrows}  For any $k>0$, there is an $N>0$ such that any map from a smooth curve $B\into Z^D(p)$ that does not factor through a Hecke curve has $B\cdot F_Z 
> k$ as long as $p>N$.
\end{proposition}
\begin{remark}For the following we work in the category of orbifold curves.  $X^D$ is naturally an orbifold curve whose orbifold points are the abelian surfaces  with extra automorphisms.  $X^D(p)$ likewise is naturally an orbifold curve, but for $p\gg 0$ its orbifold structure is trivial. Note that in this language the map $\pi:X^D(p)\into X^D$ is \'etale.
\end{remark}
\begin{proof}
We first observe that for a fixed degree $d$, the image $f_*\pi_1(B)$ of the fundamental group under a map from a smooth orbifold curve $f:B\into X^D$ of degree $d$ only depends on  the ramification profile and the monodromy around the branch points.  If we further assume the orbifold points of $B$ lie over those of $X^D$, then there are finitely many choices for this data, and therefore only finitely many possible maps $f_*:\pi_1(B)\rightarrow\pi_1(X^D)$, up to conjugacy.  It follows that for maps from orbifold curves $B\into X^D\times X^D$ of bounded bidegree $B.F$ for which the orbifold points of $B$ map to those of $X^D\times X^D$, there are only finitely many possible images of the fundamental group, again up to conjugacy.

Now, given a $k>0$ as in the Proposition, take $B\into Z^D(p)$ a smooth curve with $B.F_Z<k$, and let $B'\subset X^D\times X^D$ be its projection, with map $\alpha:B\into B'$. Note that $B'$ is naturally an orbifold curve. Consider the image $\Pi\subset \pi_1(X^D\times X^D)=\Gamma^D\times \Gamma^D$ of the fundamental group of $B'$ in  that of $X^D\times X^D$.  Note that $\Gamma^D=G(\Z)$, where $G$ is the algebraic group defined so that for a ring $R$, $G(R)=(\O_D\otimes R)^*$. For a connected component $X^D(p)_0$ of $X^D(p)$ we have
\[\pi_1(X^D(p)_0)=\Gamma^D(p)=\ker(G(\Z)\into G(\F_p))\]

We can view the map $B'\into X^D\times X^D$ as a variation of Hodge structures, in which case a theorem of Andre-Deligne \cite[Theorem 1]{andre} implies that $\Pi$ is Zariski dense in the Mumford-Tate group $G\times G$ unless $B$ factors through a Hecke curve.  By a theorem of Nori \cite[Theorem 5.1]{nori}, every Zariski dense subgroup of $G$ surjects onto $G(\F_p)$ for $p\gg 0$. Since by the above argument the number of these subgroups is finite up to conjugacy, $N$ can be chosen large enough so that we may assume $\Pi$ surjects onto $G(\F_p)\times G(\F_p)$.

To finish, $\alpha$ has degree $|G(\F_p)|$ since the inverse image of $B'$ in $Z^D(p)$ is irreducible.  The composition $B\into Z^D(p)\into X^D$ on the one hand factors through $B'$ but on the other hand has bounded degree, and we have a contradiction for $p$ large enough.

\end{proof}

\begin{remark}

In the above proof, one can avoid discussing orbifold points by manually removing the finitely many points in $X^D$ corresponding to abelian surfaces with extra automorphisms.

\end{remark}

Note that this immediately allows us to conclude

\begin{corollary}\label{easygenus}  Suppose $X^D$ has genus $g(X^D)>1$.  Then for any $k>0$, there is an $N>0$ such that any map from a smooth curve $B\into Z^D(p)$ of genus $g(B)<k$ must factor through a Hecke curve, provided $p>N$.
\end{corollary}
\begin{proof}  One of the projections $B\into X^D$ has degree at least $\frac{1}{2}(B\cdot F_Z)$.  Now use Riemann--Hurwitz and the Proposition.
\end{proof}
\section{Proof of Main Theorem}\label{mainsect}
We now prove Theorem \ref{main}. For a non-Hecke curve $B\into Z^D(p)$, let $C$ be the normalization of a component of the preimage of $B$.  The key point is that Proposition \ref{degreegrows} bounds the genus of $C$ from below, while Corollary \ref{ramo} bounds the genus from above in an asymptotically smaller way.  

\begin{theorem}\label{bigthm}
For any $k>0$, there exists an $N>0$ such that any smooth curve $B\into Z^D(p)$ of genus $g(B)<k$ must factor through a Hecke curve, provided $p>N$.
\end{theorem}
\begin{proof}  Suppose $B$ does not factor through a Hecke divisor.  Let $C\into X^D(p)\times X^D(p)$ be the normalization of a connected component of the preimage of $B$ in $X^D(p)\times X^D(p)$, and let $\alpha:C\into B$ be the map to $B$.  Fix a connected component $X^D(p)_0$ of $X^D(p)$, and identify all other components with $X^D(p)_0$.  $C$ then lands in a connected component identified with $X^D(p)_0\times X^D(p)_0$.  Note that
\[K_{X^D(p)_0\times X^D(p)_0}=(2g(X^D(p)_0)-2)F_{X^D(p)_0\times X^D(p)_0}\]
where $F_{X^D(p)_0\times X^D(p)_0}$ is the sum of the fibers, \emph{i.e.} $X^D(p)_0\times \pt$ and its flip. 
If $\pi:C\into X^D(p)_0$ is the projection onto a factor with largest degree, then on the one hand Riemann-Hurwitz applied to $\pi$ yields
\begin{equation}\label{first}g(C)\geq \frac{1}{2}(C\cdot F_{X^D(p)_0\times X^D(p)_0})(g(X^D(p)_0)-1))=\frac{1}{4}(C\cdot K_{X^D(p)_0\times X^D(p)_0})\end{equation}
On the other hand, by Corollary \ref{ramo}, as $p\into \infty$,
\[
\mult_{\CM}(C)=o(C\cdot K_{X^D(p)_0\times X^D(p)_0})\]
Thus, applying Riemann-Hurwitz to $\alpha$,
\begin{align}
g(C)&\leq 1+(\deg \alpha)(g(B)-1)+\frac{1}{2}\mult_{\CM}(C)\notag\\
&\leq 1+(\deg \alpha)(g(B)-1)+o(C\cdot K_{X^D(p)_0\times X^D(p)_0})\label{second}
\end{align}
By Proposition \ref{degreegrows} we know $(C\cdot K_{X^D(p)_0^2})/\deg\alpha=B\cdot F_Z\into\infty$ as $p\into \infty$, so after dividing equations \eqref{first} and \eqref{second} by $\deg\alpha$, for $p$ large enough we obtain the result.
\end{proof}
\begin{remark}\label{gonality}As mentioned in the introduction, the techniques of \cite{BT2} can be used to show that the constant $N$ in the theorem can be taken to only depend on the gonality of $B$.  Briefly, if $B$ is $d$-gonal, then we obtain a map $\P^1\into \Sym^d Z(p)$.  By proving repulsion results akin to those of Section 3 for the diagonals in $(X^D(p)\times X^D(p))^d$ (see \cite[Proposition 18]{BT2}), we obtain multiplicity estimates for the pull back $C$ of $B$ to $(X^D(p)\times X^D(p))^d$ along those diagonals (see \cite[Proposition 30]{BT2}).  The map $C\into B$ ramifies only when $C$ passes through diagonals or $\CM$ points, and by an argument similar to the proof above (see also \cite[Proposition 31]{BT2}) it follows that $\Sym^d Z(p)$ has no rational curves for large enough $p$.
\end{remark}
\begin{remark}For each fixed (large enough) $p$, it is easy to deduce the same result over $\bar{\F}_\ell$ for sufficiently large $\ell$ by a standard argument.
\end{remark}
\bibliography{biblio}

\begin{thebibliography}{And92}

\bibitem[And92]{andre}
Y.~Andr{\'e}.
\newblock Mumford-{T}ate groups of mixed {H}odge structures and the theorem of
  the fixed part.
\newblock {\em Compositio Math.}, 82(1):1--24, 1992.

\bibitem[BS94]{BuSa}
P.~Buser and P.~Sarnak.
\newblock On the period matrix of a {R}iemann surface of large genus.
\newblock {\em Invent. Math.}, 117(1):27--56, 1994.
\newblock With an appendix by J. H. Conway and N. J. A. Sloane.

\bibitem[BT14]{BT2}
B.~Bakker and J.~Tsimerman.
\newblock $p$-torsion monodromy representations of elliptic curves over
  geometric function fields.
\newblock \href{http://arxiv.org/abs/1403.7168}{\texttt{arXiv:1403.7168}},
  2014.

\bibitem[Cla]{clark}
P.~L. Clark.
\newblock {\em Lectures on {S}himura curves 9}.
\newblock \url{http://math.uga.edu/~pete/SC9-Orders.pdf}.

\bibitem[Cla03]{shimura}
P.~L. Clark.
\newblock {\em Rational points on Atkin-Lehner quotients of Shimura curves}.
\newblock PhD thesis, Harvard University Cambridge, Massachusetts, 2003.

\bibitem[Elk98]{elkies}
N.~D. Elkies.
\newblock Shimura curve computations.
\newblock In {\em Algorithmic {N}umber T{}heory}, pages 1--47. Springer, 1998.

\bibitem[Her91]{hermann}
C.~F. Hermann.
\newblock Modulfl\"achen quadratischer {D}iskriminante.
\newblock {\em Manuscripta Math.}, 72(1):95--110, 1991.

\bibitem[HT02]{hwangto1}
J.~Hwang and W.~To.
\newblock Volumes of complex analytic subvarieties of {H}ermitian symmetric
  spaces.
\newblock {\em American Journal of Mathematics}, 124(6):1221--1246, 2002.

\bibitem[HT12]{hwangto2}
J.~Hwang and W.~To.
\newblock Injectivity radius and gonality of a compact {R}iemann surface.
\newblock {\em American Journal of Mathematics}, 134(1):259--283, 2012.

\bibitem[KS98]{kani}
E.~Kani and W.~Schanz.
\newblock Modular diagonal quotient surfaces.
\newblock {\em Mathematische Zeitschrift}, 227(2):337--366, 1998.

\bibitem[MG78]{mazur}
B.~Mazur and Appendix by~D. Goldfeld.
\newblock Rational isogenies of prime degree.
\newblock {\em Inventiones mathematicae}, 44(2):129--162, 1978.

\bibitem[Mil97]{milne}
J.~S. Milne.
\newblock {\em Class {F}ield {T}heory}.
\newblock \url{http://www. math. lsa. umich. edu/jmilne}, 1997.

\bibitem[Nor87]{nori}
M.~V. Nori.
\newblock On subgroups of {${\rm GL}_n({\bf F}_p)$}.
\newblock {\em Invent. Math.}, 88(2):257--275, 1987.

\bibitem[Voi09]{voight}
John Voight.
\newblock Shimura curves of genus at most two.
\newblock {\em Math. Comp.}, 78(266):1155--1172, 2009.

\end{thebibliography}
\bibliographystyle{alpha}
\end{document}